\documentclass[11pt,a4paper]{amsart}
\usepackage[all]{xy}
\usepackage{theoremref}
\usepackage{graphicx}
\pdfoutput=1

\numberwithin{equation}{section}

\title{On the c-vectors of an acyclic cluster algebra}
\author{ALFREDO N\'AJERA CH\'AVEZ}
\address {Universit\'e Paris Diderot -- Paris 7\\
          Institut de Math\'ematiques de Jussieu\\
                   UMR 7586 du CNRS\\
                   Case 7012\\
                   B\^atiment Chevaleret\\
                   75205 Paris Cedex 13\\
           France}
\email {najera@math.jussieu.fr}
\newcommand{\ie}{{\em i.e. }}
\newcommand{\cf}{{\em cf.}\ }

\newtheorem{theorem}{Theorem}

\newtheorem{proposition}[theorem]{Proposition}

\newtheorem{lemma and definition}[theorem]{Lemma and Definition}

\theoremstyle{definition}

\newtheorem{remark}[theorem]{Remark}
\newtheorem{example}[theorem]{Example}
\newtheorem{definition}[theorem]{Definition}

\newcommand{\opname}[1]{\operatorname{\mathsf{#1}}}

\renewcommand{\mod}{\opname{mod}\nolimits}

\newcommand{\Mod}{\opname{Mod}\nolimits}

\newcommand{\add}{\opname{add}\nolimits}

\newcommand{\dimv}{\underline{\dim}\,}

\newcommand{\Z}{\mathbb{Z}}
\newcommand{\T}{\mathbb{T}}

\newcommand{\Hom}{\opname{Hom}}
\newcommand{\End}{\opname{End}}

\newcommand{\Ext}{\opname{Ext}}

\newcommand{\ca}{{\mathcal A}}
\newcommand{\cb}{{\mathcal B}}
\newcommand{\cc}{{\mathcal C}}
\newcommand{\cd}{{\mathcal D}}

\begin{document}
	
\maketitle

\begin{abstract}
We prove that the set of $\mathbf{c}$-vectors of the cluster algebra associated to an acyclic quiver $Q$ coincides with the set of real Schur roots and their opposites in the root system associated to $Q$.
\end{abstract}

\renewcommand{\thefigure}{\thechapter.\arabic{figure}}
\vspace{5mm}

\section{Introduction} In the theory of cluster algebras, a prominent role is played by two families of integer vectors, namely the \textbf{c}- and the \textbf{g}-vectors. They were first introduced in \cite{Clusters 4} in order to parametrize  (respectively) the coefficients and the cluster variables of a (geometric) cluster algebra. In  \cite{Nakanishi Zelevinsky} the authors showed that both families were closely related provided that the \textbf{c}-vectors satisfy the \emph{sign-coherence} property, \ie each \textbf{c}-vector has either all its entries nonnegative or all its  entries nonpositive. Moreover, many important conjectures about cluster algebras can be proved to be true if this last condition holds. The sign-coherence of the \textbf{c}-vectors was proved in \cite{QP 2} for the case of skew- symmetric exchange matrices, using decorated representations of  quivers with potentials (see \cite{QP 1}). Alternative proofs were given in \cite{Plamondon} and \cite{Nagao}. For acyclic quivers, the clusters of \textbf{c}-vectors were characterized in \cite{Speyer Thomas}. We know that \textbf{c}-vectors are always dimension vectors of indecomposable rigid modules (over an appropriate algebra), see section $8$ of \cite{Nagao}, \cf \thref{c-vecs are roots} below. In the present note, we show the other inclusion, \ie the set of positive \textbf{c}-vectors associated to an acyclic quiver $Q$ coincides with the set of real Schur roots in the root system associated to $Q$. This result can also be obtained \cite{Thomas} using the approach presented in \cite{Speyer Thomas}. A description of the \textbf{c}-vectors for general quivers seems to be unknown. A non acyclic example is computed in \cite{Najera}.\\
Section \ref{reminders} is devoted to recall the reader the constructions used along remind section \ref{main result} to prove our main result. In section \ref{examples} we present some interesting examples concerning possible generalizations of our main result.
\subsection*{Acknowledgments}
This work is part of my PhD thesis, supervised by Professor Bernhard Keller. I would like to thank him for his guidance and patience. The final version of this note was written at the MSRI, Berkeley, during the 2012 fall semester. The author is deeply grateful to this institution for providing ideal working conditions. 

\section{Reminders} \label{reminders}

Unless otherwise stated $Q$ is an arbitrary finite quiver with $n$ vertices and $k$ is an algebraically closed field.

\subsection{Dimension vectors and root systems} Denote by $kQ$ the \emph{path algebra} associated to $Q$ and denote by $\mod(kQ)$ the category of finite dimensional right $kQ$-modules (see \cite{ASS} for background material). The category $\mod(kQ)$ is equivalent to the category of finite-dimensional representations of $Q^{\text{op}}$ over $k$. We write  $\dimv M$ for the dimension vector of the representation corresponding to a module $M\in \mod (kQ)$. 

Dimension vectors are closely related with generalized root systems (which are associated to arbitrary quivers, see \cite{Kac} for details).
\begin{theorem} (\cite{Kac})
For any integer vector $\underline{v}$, there is an indecomposable $kQ$-module $M$ with $\dimv(M)=\underline{v}$ if and only if $\underline{v}$ is a positive root in the root system associated to $Q$.
\end{theorem}
\begin{definition}
A \emph{real Schur root} associated to $Q$ is the dimension vector of an indecomposable $kQ$-module without self-extensions (\emph{rigid}). These vectors are independent of $k$ (see \cite[Theorem 1]{Crawley-Boevey}). We denote by $\Phi^{re,Sch}$ the set of real Schur roots.
\end{definition}

\subsection{Tilting theory and the cluster category} In this section, we assume that $Q$ is acyclic. Denote by $\ca$ the category $\mod(kQ)$. 
\begin{definition}
We say that  a module $T \in \ca $ is \emph{tilting} if and only if $T$ is \emph{rigid}, \ie $\Ext^1_{\ca}(T,T)=0$, and $T$ is the direct sum of $n$ pairwise non isomorphic indecomposable modules. For a tilting module $T$, we denote by $B$ the algebra $\End_{\ca}(T)$ and by $\cb$ the category $\mod(B)$.
\end{definition}

 Let $C$ be a category of the form $\ca$ or $\cb$. Then the Grothendieck gropu $K_0(C)$ admits an unique bilinear form $\langle \cdot, \cdot \rangle$ for which the basis of the simples is dual to the basis of the indecomposable projectives $\lbrace P_1, \dots, P_n \rbrace$ in the sense that $\langle [P_i],[S_j]\rangle=\delta_{ij}$ for $1\leq i,j\leq n$. We call an isomorphism of Grothendieck groups an \emph{isometry} if it respects the bilinear form. We denote by $\cd^b(C)$ the bounded derived category of $C$ and by $\Sigma$ its canonical shift (or suspension) functor. The natural inclusion from $C$ to its derived category induces an isomorphism between $K_0(C)$ and $K_0(\cd^b(C))$, the Grothendieck group of $K_0(\cd^b(C))$ as a \emph{triangulated} category (\cf \cite{Happel triang}). The following theorem will be very useful in the sequel.

\begin{theorem} (\cite[Section 2]{Happel})\thlabel{isometry}
Let $T$ be a tilting module. Then the functor 
\begin{equation}
-\otimes_B^{L}T:\cd^b(\cb)\rightarrow \cd^b(\ca)
\end{equation}
is a triangle equivalence which induces an isometry of Grothendieck groups $K_0(\cb)\rightarrow K_0(\ca)$.
\end{theorem}

\begin{definition} (\cite{BMRRT})
Let $\cc$ be the \emph{cluster category} associated to $Q$, that is, the orbit category $\cd^b(\ca)/\tau^{-1}\circ\Sigma$, where $\tau$ denotes the AR-translation of $\cd^b(\ca)$. The category $\cc$ admits a canonical triangulated structure \cite{Keller triang} whose suspension functor is induced by $\Sigma$. An object $X$ in $\cc$ is called \emph{cluster tilting} if $\Hom_{\cc}(X,\Sigma X)=0$ and $X$ is the sum of $n$ pairwise non isomorphic indecomposable objects. 
\end{definition}

\begin{remark}
By construction, there is a natural functor from of $\mod(kQ)$ to $\cc$ which is faithful but not full in general. By \cite{BMRRT} we know that the image of every tilting module is a cluster tilting module.
\end{remark}

\subsection{Derived categories for dg-algebras} \label{derived categories}
We follow the approach of \cite{Keller dg}. Throughout this subsection, $A$ is a dg-algebra over $k$ with differential $d_A$. That is a $\Z$-graded $k$-algebra or equipped with a degree $1$ $k$-linear differential. A dg-module $M$ over $A$ is a graded $A$-module endowed with a differential $d_M$ satisfying
\begin{equation*}
d_M(ma)=d_M(m)a+(-1)^{|m|}md_A(a)
\end{equation*}
for every homogeneous element $m$ in $M$ of degree $|m|$, and every $a$ in $A$.
Given two dg $A$-modules $M$ and $N$, the \emph{morphism complex} is the graded vector space $\mathcal{H}om_A(M,N)$ whose $i$-th component is the subspace of the product $\prod_{j\in\Z} Hom_k(M^j,N^{j+i})$ consisting of the morphisms $f$ such that $f(ma) = f(m)a$, for all $m$ in $M$ and all $a$ in $A$, and whose differential $d$ is given by
\begin{equation*}
d(f)=f\circ d_M + (-1)^{|f|}d_N \circ f
\end{equation*}
for a homogeneous morphism $f$ of degree $|f|$.
\\
Let $\mathcal{H}(A)$ be the \emph{homotopy category} of dg $A$-modules. It has as objects the dg $A$-modules and as morphisms the $0$-th homology groups of the morphism complexes, \ie $\mathcal{H}(A)(M,N)=H^0(\mathcal{H}om_A(M,N)))$. The \emph{derived category} $\cd (A)$ is the localization \cite{Gabriel Zisman} of $\mathcal{H}A$ with respect to  the full subcategory of acyclic dg $A$-modules. In other words, we are adding formal inverses to those morphisms of complexes inducing isomorphisms in all homology groups.

\begin{remark} \thlabel{triangulated}
$\mathcal{D}(A)$ is an additive category. Moreover it is always triangulated \cite{Happel triang}, with \emph{suspension functor} $\Sigma$ given by the degree shift on complexes. 
\end{remark}

\begin{remark} \thlabel{classical setting}
The derived category $\cd(\Mod A)$ of the abelian category $\Mod A$ of all $A$-modules, where $A$ is a $k$-algebra, is a classical object (\cf \cite{Grothendieck}). However, it can be obtained from this more general setting by thinking of $A$ as a dg algebra concentrated in degree $0$.\end{remark}

\subsection{Quivers with potentials and the Ginzburg dg algebra}
For this section, we assume $Q$ has no loops nor two-cycles. Let $\widehat{kQ}$ be the \emph{completed path algebra} associated to $Q$, \ie the completion of $kQ$ with respect to the \emph{path length}. As a topological algebra, $\widehat{kQ}$ is endowed with the $\mathfrak{m}$-adic topology, where $\mathfrak{m}$ is the ideal generated by the arrows in $Q$. A \emph{potential} in $Q$ is a (possibly infinite) linear combination of cyclic paths in $Q$, considered up to \emph{cyclic equivalence} and whose summands are pairwise cyclically inequivalent (\cf \cite[Section 2]{QP 1}). Potentials can be thought of as elements of the space
\begin{equation*}
Pot(Q)=\widehat{kQ}/C,
\end{equation*}
where $C$ is the closure in $\widehat{kQ}$ of the subspace $[\widehat{kQ},\widehat{kQ}]$. The only nonzero potentials considered in this note are the \emph{reduced} ones, \ie those which only involve cycles of length strictly greater than $2$. For each arrow $\alpha$ of $Q$, we define the \emph{cyclic derivative}
\begin{equation*}
\partial_{\alpha}:Pot(Q)\longrightarrow \widehat{kQ} 
\end{equation*}
as the continuous linear map taking the equivalence class of a cycle $c$ to the sum
\begin{equation*}
\sum_{c= u \alpha v} vu
\end{equation*}
taken over all possible cyclic paths  $u\alpha v$ equal to $c$.
\begin{definition}
The \emph{Jacobian algebra} $J(Q,W)$ associated to a \emph{quiver with potential} $(Q,W)$ (QP for short) is the quotient
\begin{equation*}
J(Q,W)=\widehat{kQ}/\partial(W).
\end{equation*}
where $\partial(W)$ denotes the closure of the ideal generated by the elements $\partial_{\alpha}(W)$ as $\alpha$ ranges through all the arrows of $Q$.
\end{definition}
The mutation operation on quivers admits an extension to the class of reduced quivers with potential up to an appropriate equivalence, namely the \emph{right equivalence} (\cf Sections 4 and 5 of \cite{QP 1}). In contrast to quiver mutation, iterated QP mutation can produce $2$-cycles in the quivers. We call a potential \emph{non-degenerate} if this is not the case. 

Now fix a quiver with potential $(Q,W)$. We recall the construction of the (completed) \emph{Ginzburg dg algebra} $\Gamma_{Q,W}=\Gamma$ (for short) associated to $(Q,W)$ (recall that a \emph{dg algebra} over $k$ is  a $\Z$-graded $k$-algebra or equipped with a degree $1$ $k$-linear differential).  We let $\overline{Q}$ be the graded quiver obtained from $Q$ as follows:
\begin{itemize}
\item $\overline{Q}$ has the same vertices as $Q$;
\item the set of arrows of degree 0 in $\overline{Q}$ is the set of arrows of $Q$;
\item the set of arrows of degree $-1$ in $\overline{Q}$ is formed by arrows $\alpha^{\ast}:j\rightarrow i$, for each arrow $\alpha:i\rightarrow j$ in $Q$;
\item the set of arrows of degree $-2$ in $\overline{Q}$ is formed by loops $t_i:i\rightarrow i$, one for each vertex $i$ of $\overline{Q}$;
\item these are the only arrows.
\end{itemize}
The underlying graded algebra of $\Gamma$ is the graded completion of the graded path algebra $k\overline{Q}$ with respect to the \emph{path degree} (lazy paths introduced below are considered as of degree $0$). The differential $d$ of $\Gamma$ is the unique continuous linear differential acting as follows on the arrows:
\begin{itemize}
\item $d(\alpha)=0$ for the arrows $\alpha$ of degree $0$;
\item $d(\alpha^{\ast})=\partial_{\alpha}(W)$ for the arrows $\alpha^{\ast}$ of degree $-1$;
\item $d(t_1)=e_i(\sum_a[a,a^{\ast}])e_i$ for each loop $t_i$ of degree $-2$. Here the sum is taken over all arrows of degree $0$ and $e_i$ is the (\emph{lazy}) path of length zero that stays at vertex $i$.
\end{itemize}
This definition is slightly different from the one  first introduced in \cite{Ginzburg} see \cite[Remark 2.7]{Keller-Yang}.
\begin{remark}
The zeroth homology  $H^0(\Gamma_{Q,W})$ is canonically isomorphic to $J(Q,W)$.
\end{remark}
We refer to \cite{Keller-Yang} for the effect of the mutation of QPs on the associated Ginzburg algebras.

%%%%%%
%Introduce c-vectors
%%%%%%
\section{Main result} \label{main result}
\begin{theorem} \thlabel{main theorem}
If $Q$ is acyclic, then the set of $\mathbf{c}$-vectors associated to $Q$ is equal to $\Phi^{re,Sch}\cup-\Phi^{re,Sch}$.
\end{theorem}
We deduce the theorem from the following propositions. 
\begin{proposition}
\label{image under F}
If $Q$ is acyclic and  $T=\bigoplus T_i$ is a tilting module in $\mod{kQ}$ and $B$ the endomorphism algebra $\End_{kQ}(T)$, then the image of each simple $B$-module under the isometry induced by the equivalence $F=-\otimes_B^{L}T:\cd^b(\cb)\rightarrow \cd^b(\ca)$ (see \thref{isometry}) is a $\mathbf{c}$-vector.
\end{proposition}
%and let $T=\bigoplus T_i$ be a tilting module of $\ch$ and
%Suppose without lost of generality that $T$ is the image of a tilting module $\tilde{T}$ under the canonical projection $\mod(KQ)\rightarrow \cc_Q$. 
%denote by $B$ the endomorphism algebra $\End_{KQ}(T)$. 
\begin{proof}
Let $\cc_Q$ be the cluster category associated to $Q$. Let $\tilde{T}=\bigoplus \tilde{T}_i$ be the image of $T$ in $\cc_Q$. Thus, $\tilde{T}$ is a cluster tilting object. Denote by $(g^{\tilde{T}}_{ij}) $ (resp. $(c^{\tilde{T}}_{ij})$, with $1\leq i, j\leq n $), the $\mathbf{g}$-matrix (resp. $\mathbf{c}$-matrix) associated to $\tilde{T}$. Is easy to see that under the equivalence $F$, the class of $B$ is mapped onto the class of $T$, and the classes of the indecomposable projective modules $P^B_i$ are mapped onto the indecomposable factors $T_i$ of $T$. By \cite[Corollary 6.8]{Keller}, we know that $[T_i]=\Sigma g^{\tilde{T}}_{ji}[P^A_j]$. This implies that $F[S^B_i]=\Sigma c^{\tilde{T}}_{ji}[S^A_i]$. Indeed, let $F[S_i^B]=\Sigma n_{ji}[S^A_j]$. Then
\begin{align*}
\delta_{ij} =&\ \langle F[P^B_i],F[S^B_j] \rangle \\
=&\ \langle \Sigma g^{\tilde{T}}_{ki}[P^A_k], \Sigma n_{kj}[S^A_k] \rangle\\
=& \Sigma g^{\tilde{T}}_{ki}n_{kj},
\end{align*}
and by \cite[Theorem 1.2]{Nakanishi Zelevinsky}, we obtain $n_{ij}=c^{\tilde{T}}_{ij}$ for $1\leq i,j \leq n$.
\end{proof}
% Thus, in order to obtain a real Schur root as a $\mathbf{c}$-vector, it would suffie to complet every rigid indecomposable module $M$, to a cluster tilting object $T=M\oplus G$ with $\Hom_{\cc}(G,M)=0$. Indeed, in such case, the vertex corresponding to $M$  in the quiver of $\End_{\cc}(T)$ is a source, \ie, M is a projective simple.
\begin{proposition}
Each root in $\Phi^{re,Sch}\cup-\Phi^{re,Sch}$ is a $\mathbf{c}$-vector.
\end{proposition}
\begin{proof}
Let $M$ be a non injective rigid indecomposable $kQ$-module and let $H=kQ$. We consider the (dual version of the) Bongartz exact sequence associated to $M$, \ie the universal extension
\begin{equation*}
0 \rightarrow M^r \rightarrow G \rightarrow DH\rightarrow 0
\end{equation*}
of $DH$ by an object of $\add(M)$ (see \cite{Bongartz} or \cite[VI.2.]{ASS}). We know that $T:=G\oplus M$ is a tilting module. Moreover, $G$ is a projective generator of the abelian category $M^{\bot}=\lbrace N : \Hom(M,N)=0=\Ext^1(M,N)\rbrace$. In particular, the vertex corresponding to $M$ of the quiver of $B=\End_{kQ}(T)$ is a source. If we denote by $S_M$ the simple projective $B$-module associated to this vertex, then $F[S_M]=M$, and therefore, $\dimv(M)$ is a $\mathbf{c}$-vector by Proposition \ref{image under F}. For an injective indecomposable $I$, we chose a tilting complex $T$ by completing $I$ into a section with source $I$ of the AR-quiver of $\cd^b(A)$. For $B=\End(T)$, we have a triangle equivalence $F:\cd^b(B)\rightarrow \cd^b(A)$ taking a simple projective to $I$. As above, we conclude that $\dimv I$ ia a $\mathbf{c}$-vector.
\end{proof}
The remaining inclusion is immediate from the following general fact.
\begin{theorem}(\cite{Nagao})\thlabel{c-vecs are roots}
Let $Q$ be a quiver without loops nor $2$-cycles and $W$ a non degenerate potential on $Q$. Then each positive $\mathbf{c}$-vector of $Q$ is the dimension vector of a finite dimensional rigid module with endomorphism algebra $k$ over the Jacobian algebra of $(Q,W)$.
\end{theorem}
\begin{proof}
We follow the approach of \cite{Nagao} as presented in section $7.7$ of \cite{Keller 2}. We denote by $J(Q,W)$ the Jacobian algebra associated to the opposite quiver $(Q^{op},W^{op})$ and by $\Gamma$ the corresponding Ginzburg dg algebra. We consider $\cd_{fd}(\Gamma)$ the full sub-category of $\cd(\Gamma)$ formed by the dg modules whose homology is of finite total dimension. The homology induces a natural $t$-structure in $\cd_{fd}(\Gamma)$, which admits $\ca=\mod J(Q,W)$ as heart. Let $\T_n$ denote the $n$-regular tree. If we assign to a fixed vertex $t_0\in \T_n$ the quiver with potential $(Q,W)$, then, by iterated mutation, we can associate a quiver with potential $(Q(t),W(t))$ and a $\mathbf{c}$-matrix $C(t)$ to any vertex $t \in \T_n$. We write $\Gamma (t)$ for the Ginzburg dg algebra associated to the opposite of $(Q(t),W(t))$. We denote by
\begin{equation*}
\Phi(t):\cd(\Gamma(t))\rightarrow \cd(\Gamma)
\end{equation*}
the triangle equivalence constructed in section $7.7$ of \cite{Keller 2}. By parts $a)$ and $b)$ of Theorem 7.9 of \cite{Keller 2}, \cf also Remark 8.2 of \cite{Nagao}, we have
\begin{itemize}
\item[a)] the image $S_j(t)$ of the simple $S_j$ under $\Phi (t)$ lies in $\ca$ or $\Sigma^{-1}\ca$,
\item[b)] the $\mathbf{c}$-vector $C(t)e_j$ equals $\dimv S_j(t)$.
\end{itemize}
The object $S_j(t)$ has endomorphism algebra $k$ and does not have self-extensions since the object $S_j$ has these properties in $\cd(\Gamma(t))$. Thus, we find: each positive $\mathbf{c}$-vector is the dimension vector of a rigid indecomposable module over $J(Q,W)$. 
\end{proof}

\section{examples} \label{examples}
We present some examples beyond the scope of our main result. An interesting question is to ask whether the inclusion in \thref{c-vecs are roots} is an equality. The following example shows that this is not the case in general.

\begin{example}
Let Q be the Markov quiver 
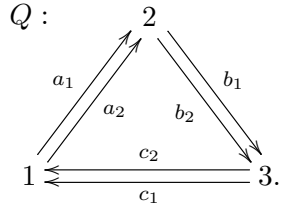
\begin{equation*}
\begin{xy} 0;<.6pt,0pt>:<0pt,-.6pt>:: 
(0,0) *+{Q:} ="0",
(0,100) *+{1} ="1",
(150,100) *+{3.} ="3",
(75,0) *+{2} ="2",
"3", {\ar_{c_2}@<-0.5ex>"1"},
"3", {\ar^{c_1}@<+0.5ex>"1"},
"1", {\ar_{a_2}@<-0.5ex>"2"},
"1", {\ar^{a_1}@<+0.5ex>"2"},
"2", {\ar_{b_2}@<-0.5ex>"3"},
"2", {\ar^{b_1}@<+0.5ex>"3"},
\end{xy}
\end{equation*}
The potential $W=c_1b_1a_1+c_2b_2a_2-c_1b_2a_1c_2b_1a_2$ makes $J(Q,W)$ finite dimensional as shown in \cite[Example 8.2]{Labardini}. We know that the dimension vectors of the indecomposable projective modules are $(4,4,4)$ (see \cite[Example 8.6]{QP 1}). By the description in \cite{Najera} of the \textbf{c}-vectors of this quiver, we know that $(4,4,4)$ is not a \textbf{c}-vector.
\end{example}

Another possible direction for generalizing \thref{main theorem} is letting $Q$ be mutation equivalent to an acyclic quiver. In an upcoming paper \cite{Najera 2}, it will be proved that if $Q$ is mutation equivalent to a Dynkin quiver then the positive \textbf{c}-vectors associated to $Q$ are precisely the dimension vectors of indecomposable rigid modules over $J(Q,W)$ for a generic potential $W$. This is no longer true in more general settings, as the following example shows.

\begin{example}
Let $Q$ be the quiver 
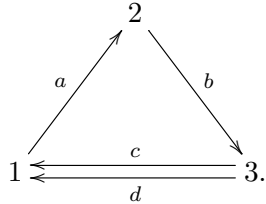
\begin{equation*}
\begin{xy} 0;<.6pt,0pt>:<0pt,-.6pt>:: 
(0,100) *+{1} ="1",
(150,100) *+{3.} ="3",
(75,0) *+{2} ="2",
"3", {\ar_{c}@<-0.5ex>"1"},
"3", {\ar^{d}@<+0.5ex>"1"},
"1", {\ar^a"2"},
"2", {\ar^b"3"},
\end{xy}
\end{equation*}
This quiver is of type $\tilde{A}_2$ and $W=cba$ is a generic potential. The \textbf{c}-vectors associated to $Q$ are described in \cite{Cerulli-Irelli}. We can verify that $(1,2,1)$ is not a \textbf{c}-vector, however it is the dimension vector of the indecomposable projective $P_2$ over $\mod(J(Q,W))$.
\end{example}
% we place ourselves in a more general setting. We follow \cite{Keller-Yang} and \cite{QP 1}. Let $(Q,W)$ be a quiver (without loops nor 2-cycles) endowed with a generic potential $W$. We denote by $J(Q,W)$ the Jacobian algebra associated to the opposite quiver $(Q^{op},W^{op})$ and by $\Gamma$ the corresponding Ginzburg dg algebra. We consider $\cd(\Gamma)$ the derived category of $\Gamma$  and $\cd_{fd}(\Gamma)$ the full sub-category formed by the dg modules whose homology is of finite total dimension. There is a natural $t$-structure in $\cd_{fd}(\Gamma)$, which admits $\ca=\mod J(Q,W)$ as heart. Let $\T_n$ denote the $n$-regular tree. If we assign to a fixed vertex $t_o\in \T_n$ the quiver with potential $(Q,W)$, then, we can associate by iterated mutation a quiver with potential $(Q(t),W(t))$ and a $\mathbf{c}$-matrix $C(t)$ to any vertex $t \in \T_n$. We write $\Gamma (t)$ for the Ginzburg dg algebra associated to the opposite of $(Q(t),W(t))$. By \cite[Theorem 3.2]{Keller-Yang} there is a triangle equivalence $\phi(t):\cd(\Gamma (t))\rightarrow \cd(\Gamma) $ with the following property (see \cite[Theorem 7.9]{Keller 2}): the jth column of the $\mathbf{c}$-matrix $C(t)$ contains the coordinates of $[S_j(t)]$ (the class of the simple $J(Q(t),W(t))$-module corresponding to the vertex $i$) in the basis $[S_1], \ldots , [S_n]$ of $K_0(\cd_{fd}(\Gamma))$.


\begin{thebibliography}{1}

\bibitem{ASS}

I. Assem, D. Simson, A. Skowronski, \emph{Elements of the Representation theory of Associative Algebras 1: Techniques of Representation Theory}. London Mathematical Society Student Texts 65, Cambridge University Press, Cambridge, 2006.

\bibitem{Bongartz}

K. Bongartz, Tilted algebras, in Proc. ICRA III (Puebla, 1980), Lecture Notes in Math. No. 903, Springer-Verlag, Berlin, Heidelberg, New York, 1981, pp. 26--38.

\bibitem{BMRRT}
A. B. Buan, R. J. Marsh, M. Reineke, I. Reiten and G. Todorov, \emph{Tilting theory and cluster combinatorics}, Advances in Mathematics \textbf{204 (2)} (2006), 572--618.

\bibitem{Cerulli-Irelli}
G. Cerulli Irelli, \emph{Structural theory of rank three cluster algebras of affine type}, Unpoblished PhD thesis. Universit\`a degli studi di Padova (2008).

\bibitem{Crawley-Boevey}

W. Crawley-Boevey, \emph{Rigid integral representations of quivers}. Representation Theory of Algebras, CMS Conference Proceedings, \textbf{18} (1996), pp. 155–163 Providence, RI.

\bibitem{QP 1}

H. Derksen, J. Weyman and A. Zelevinsky, \emph{Quivers with potentials and their representations I: Mutations} Selecta Math. \textbf{14} (2008), no. 1, 59--119.

\bibitem{QP 2}

\bysame, \emph{Quivers with potentials and their representations II: Applications to cluster algebras}, J. Amer. Math. Soc. $\mathbf{23}$ (2010), 749--790.

\bibitem{Clusters 4}

S. Fomin and A. Zelevinsky, \emph{Clusters algebras IV: Coefficients}, Composito Mathematica \textbf{143} (2007), 112--164.

\bibitem{Gabriel Zisman}
P. Gabriel and M. Zisman, \emph{Calculus of fractions and homotopy theory}, Ergebnisse der Mathematik und ihrer Grenzgebiete, Band 35, Springer-Verlag New York, Inc., New York, 1967.

\bibitem{Ginzburg}
V. Ginzburg, \emph{Calabi-Yau algebras}, arxiv:0612139v3 [math.AG].

\bibitem{Grothendieck}
A. Grothendieck, \emph{Groupes des classes des cat\'egories abeliennes et triangul\'es. Complexes
parfait}, in SGA 5, Expos\'e VIII, Springer Lecture Notes \textbf{589} (1977).

\bibitem{Happel}

D. Happel, \emph{On the derived category of a finite-dimensional algebra}, Comment. Math. Helv, \textbf{62}(1987), no. 3. 339--389.

\bibitem{Happel triang}
\bysame, \emph{Triangulated categories in the representation theory of finite-dimensional algebras}, Cambridge University Press, Cambridge, 1988.

\bibitem{Kac}
V. G. Kac, \emph{Infinite root systems, representations of graphs and invariant theory}, Invent. Math. \textbf{56} (1980), no. 1, 57--92. MR 557581 (82j:16050)

\bibitem{Keller dg}

B. Keller. \emph{Deriving DG categories}, Ann. Scient. Ec. Norm. Sup. \textbf{27} (1994), 63--102.

\bibitem{Keller triang}
\bysame, \emph{On triangulated orbit categories}, Doc. Math. \textbf{10} (2005), 551–581 (electronic).

\bibitem{Keller}

\bysame, \emph{The periodicity conjecture for pairs of Dynkin diagrams}, arxiv:1001.1531.


\bibitem{Keller 2}

\bysame, \emph{Cluster algebras and derived ctegories}, arxiv:1202.4161.

\bibitem{Keller-Yang}

B. Keller, D. Yang, \emph{Derived equivalences from mutations of quivers with potential}, Advances in Mathematics \textbf{26} (2011), 2118--2168.

\bibitem{Labardini}
D. Labardini-Fragoso, \emph{Quivers with potentials associated to triangulated surfaces}. Proc. London Mathematical Society (2009) \textbf{98} (3), 797--839.

\bibitem{Nagao}

K. Nagao, \emph{Donaldson-Thomas theory and cluster algebras}, arXiv:1002.4884 [math.AG].

\bibitem{Najera}

A. N\'ajera Ch\'avez, \emph{On the c-vectors and g-vectors of the Markov cluster algebra}, eprint, arXiv:1112.3578 [math.CO].

\bibitem{Najera 2}
\bysame, \emph{C-vectors and dimension vectors for cluster-finite quivers}, to appear.

\bibitem{Nakanishi Zelevinsky}

T. Nakanishi, A. Zelevinsky, \emph{On tropical dualities in cluster algebras}, eprint, arXiv:1101.3736, 2011.

\bibitem{Plamondon}

P. Plamondon, \emph{Cluster algebras via cluster categories with infinite-dimensional morphism spaces}, Compos. Math. \textbf{147} (2011), 1921--1954.

\bibitem{Plamondon}
\bysame, \emph{Generic bases for cluster algebras from the cluster category}, Int. Math. Res. Notices 2012; doi: 10.1093/imrn/rns102.

\bibitem{Speyer Thomas}

D. Speyer, H. Thomas, \emph{Acyclic cluster algebras revisited}, eprint. arXiv:1203.0277 [math.RT].

\bibitem{Thomas}

H. Thomas, private communication (March 2012).

\end{thebibliography}
\end{document}